\documentclass[11pt,a4paper]{amsart}

\usepackage{amsmath,amssymb,enumerate,amsthm,graphicx}
\usepackage[initials]{amsrefs}
 
\numberwithin{equation}{section}
          
\theoremstyle{plain}
\newtheorem{theorem}{Theorem}[section]
\newtheorem{proposition}[theorem]{Proposition}

\theoremstyle{definition}

\newtheorem{example}[theorem]{Example}

\newcommand{\UP}{\blacktriangle}
\newcommand{\DOWN}{\blacktriangledown}

\begin{document}

\tolerance=5000

\title[Monteiro spaces and rough sets: Models for Nelson algebras]%
{Monteiro spaces and rough sets determined by quasiorder relations: Models for Nelson algebras}

\author[J.~J{\"a}rvinen]{Jouni J{\"a}rvinen}
\address{J.~J{\"a}rvinen, Sirkankuja 1, 20810~Turku, Finland}
\email{Jouni.Kalervo.Jarvinen@gmail.com}
\urladdr{\url{http://sites.google.com/site/jounikalervojarvinen/}}

\author[S.~Radeleczki]{S{\'a}ndor Radeleczki}
\thanks{Acknowledgements: The research of the second author was carried out as
part of the TAMOP-4.2.1.B-10/2/KONV-2010-0001 project supported by the
European Union, co-financed by the European Social Fund.}
\address{S.~Radeleczki, Institute of Mathematics\\ 
University of Miskolc\\3515~Miskolc-Egyetemv{\'a}ros\\Hungary}
\email{matradi@uni-miskolc.hu}
\urladdr{\url{http://www.uni-miskolc.hu/~matradi/}}

\maketitle

\begin{abstract}
The theory of rough sets provides a widely used modern tool, and in particular, rough sets induced 
by quasiorders are in the focus of the current interest, because they are strongly interrelated with the 
applications of preference relations and intuitionistic logic.
In this paper, a structural characterisation of rough sets 
induced by quasiorders is given. These rough sets form Nelson algebras defined on algebraic lattices.
We prove that any Nelson algebra can be represented as a subalgebra of an
algebra defined on rough sets induced by a suitable quasiorder. We also show
that Monteiro spaces, rough sets induced by quasiorders and Nelson algebras
defined on  $\mathrm{T_0}$-spaces that are Alexandrov topologies can be
considered as equivalent structures, because they determine each other up to
isomorphism.
\end{abstract}

\section{Rough sets}

The theory of \emph{rough set}s introduced by Pawlak \cite{Pawl82}
can be viewed as an extension of  the classical set theory. Its fundamental 
idea is that our knowledge about the properties of the objects of a 
given universe of discourse $U$ may be inadequate or incomplete in a sense that the
objects of the universe $U$ can be observed only within the accuracy of indiscernibility relations.
According to Pawlak's original definition, an indiscernibility relation $E$ on $U$ 
is an equivalence relation (reflexive, symmetric, and transitive binary relation)
interpreted so that two elements of $U$ are $E$-related if 
they cannot be distinguished by their properties known by us. Since there is a one-to-one 
correspondence between equivalences and partitions, each indiscernibility relation 
induces a partition on $U$ such that its blocks consist of objects that are
precisely similar with respect to our information.  In this sense,  our ability to distinguish objects 
can be understood to be blurred -- we cannot distinguish individual objects, only 
groups of similar objects. But this is often the case in practice;
we may have objects that are indistinguishable by their properties, but one of them 
belongs to some set (e.g. people that have  certain disease), 
while the other one does not.

Each subset $X$ of $U$ can be approximated by two sets: 
the \emph{lower approximation} $X^\DOWN$ of $X$ consists of the $E$-equivalence classes that 
are included in $X$, and the \emph{upper approximation} $X^\UP$ of $X$ contains the $E$-classes 
intersecting $X$. The lower approximation $X^\DOWN$ can be viewed as the set of 
elements that are \emph{certainly} in $X$ and the upper approximation $X^\UP$ can be 
considered as the set of elements that \emph{possibly} belong to $X$. A consequence of
this is that the membership functions of sets become three-valued: 1 (the element belongs to the set), 
0 (the element is not in the set), $u$ (unknown borderline case: the element is simultaneously inside 
and outside the set, at some degree).

Two subsets $X$ and $Y$ of $U$ are defined to be $\equiv$-related 
if both of their approximations are the same, that is, $X^\DOWN =
Y^\DOWN$ and $X^\UP = Y^\UP$. Clearly, the relation $\equiv$ is an equivalence, 
and its equivalence classes are called \emph{rough sets}. 
Each element in the same rough set looks the same, when
observed through the knowledge given by the indiscernibility relation $E$.
Namely, if $X \equiv Y$, then exactly the same elements belong certainly and
possibly to $X$ and $Y$.

Lattice-theoretical study of rough sets was initiated by T.~B.~Iwi{\'n}ski in \cite{Iwin87}.
He pointed out that since each rough set is uniquely determined by the lower and the upper 
approximations of its members, the set of rough sets can be defined as
\begin{equation*}
\mathit{RS} = \{ (X^\DOWN,X^\UP) \mid X \subseteq U \}. 
\end{equation*}
In addition, $\mathit{RS}$ may be canonically ordered by the coordinatewise order:
\begin{equation*}
(X^\DOWN,X^\UP) \leq (Y^\DOWN,Y^\UP) \iff X^\DOWN \subseteq Y^\DOWN \mbox{ \
and \ } X^\UP \subseteq Y^\UP.
\end{equation*}

In computer science, rough sets represent a widely used modern tool; they are applied,
for instance, to approximative reasoning in feature selection problems, learning theory, and combined 
with methods of fuzzy sets or of formal concept analysis they are used in data mining also 
\cite{hassanien2007rough}.

In the literature can be found numerous studies on rough sets that are
determined by different types of relations reflecting distinguishability or
indistinguishability of the elements of the universe of discourse $U$ (see e.g.~\cite{DemOrl02}). 
If $R \subseteq U\times U$ is an arbitrary binary relation, then the lower and upper
approximations of a set $X\subseteq U$ are defined as follows.
For any $x\in U$, we denote $R(x)= \{y \in U \mid x \, R \, y \}$.
The \emph{lower approximation} of $X$ is 
\[
X^\DOWN = \{x \in U \mid R(x)\subseteq X\},
\]
and $X$'s \emph{upper approximation} is 
\[
X^\UP=\{x \in U \mid R(x) \cap X \neq \emptyset\}.
\]

If $R$ is reflexive, then $X^\DOWN \subseteq X \subseteq X^\UP$. In the
case $R$ is a \emph{quasiorder}, that is, $R$ is a reflexive and transitive binary 
relation on $U$, we have $x \, R \, y \iff R(y)\subseteq R(x)$, and 
the map $X \mapsto X^\UP$ is a topological closure operator and 
$X \mapsto X^\DOWN$ is a topological interior operator on the set $U$ 
(see \cite{Jarv07}).

Rough sets induced by quasiorders are in the focus of current interest; see 
\cite{JPR12,JarRad11, JRV09,  umadevi12}, for example. Let us denote by $\wp(U)$ 
the power set of $U$, that is, the set of all subsets of $U$. It was shown 
by J.~J\"{a}rvinen, S.~Radeleczki, and L.~Veres \cite{JRV09} that $\mathit{RS}$  is a complete 
sublattice of $\wp(U) \times\wp(U)$ ordered by the coordinatewise set-inclusion relation, which 
means that $\mathit{RS}$ is an algebraic completely distributive lattice such that
\[
\bigwedge\left\{  ( X^\DOWN, X^\UP ) \mid X\in\mathcal{H} \right\}  =
\Big ( \bigcap_{X \in\mathcal{H}} X^\DOWN, \bigcap_{X\in\mathcal{H}} X^\UP
\Big )
\]
and
\[
\bigvee\left\{  ( X^\DOWN, X^\UP ) \mid X\in\mathcal{H} \right\}  =
\Big ( \bigcup_{X\in\mathcal{H}} X^\DOWN, \bigcup_{X\in\mathcal{H}} X^\UP \Big )
\]
for all $\mathcal{H} \subseteq \wp(U)$. Since $\mathit{RS}$ is a completely distributive complete lattice, 
it is a Heyting algebra, that is, a lattice with $0$ such that for each $a,b$, there is a greatest 
element $x$ with $a \wedge x \leq b$. This element is the \emph{relative pseudo\-complement} of 
$a$ with respect to $b$, and is denoted $a \Rightarrow b$.

\emph{Constructive logic with strong negation} was introduced by Nelson \cite{Nelson49} and 
independently by Markov \cite{Markov50}. It is often called simply as \emph{Nelson logic}.
It is an extension of the intuitionistic propositional logic by \emph{strong
negation} $\sim$. The intuitive reading of ${\sim} A$ is ``a counterexample of $A$''.
As described in \cite{Vaka05}, one sentence $A$ may have many counterexamples and each of them
needs to contradict $A$. For instance, a counterexample of the sentence ``This
apple is red'' is ``This apple is green'' or ``This apple is yellow''.
The axioms of Nelson logic can be interpreted as ``algorithms'' of constructing counterexamples of 
compound sentences by means of given counterexamples of their components, and
the name \textit{strong negation} comes from the fact that the formula 
${\sim} A \to \neg A$ is a theorem of the logic. Nelson logic is axiomatized by
extending intuitionistic logic with the formulas:
\begin{enumerate}[({NL}1)]
\item ${\sim} A \to (A \to B)$  \\
\emph{(a counterexample of A contradicts A, that is, $A \wedge {\sim} A$ implies everything)}

\item ${\sim} (A \to B) \leftrightarrow A \wedge {\sim} B$ \\
\emph{(a counterexample of $A \to B$ can be constructed by the conjunction of 
$A$ with a counterexample of $B$)}

\item ${\sim} (A \wedge B)\leftrightarrow {\sim} A \vee {\sim} B$ \\
\emph{(a counterexample of a conjunction can be constructed
as a disjunction of counterexamples of its components)}

\item ${\sim} (A \vee B) \leftrightarrow {\sim} A \wedge {\sim} B$ \\
\emph{(a counterexample of a disjunction can be can be constructed
as a conjunction of counterexamples of its components)}

\item ${\sim}\,\neg A \leftrightarrow A$ \\
\emph{($A$ is a counterexample of $\neg A$)}

\item ${\sim}\,{\sim} A \leftrightarrow A$ \\
\emph{($A$ is a counterexample of a counterexample of $A$)}
\end{enumerate}

A \textit{Nelson algebra}  is a structure $\mathbb{A} = (A, \vee, \wedge, \rightarrow, {\sim}, 0, 1)$ 
such that $(A,\vee,\wedge,0,1)$ is a bounded distributive lattice and for all $a,b,c\in A$:
\begin{enumerate}[({N}1)]
\item ${\sim}\,{\sim}a  =  a$,
\item $a \leq b$  \ if and only if \   ${\sim}b \leq {\sim}a$,
\item $a \wedge {\sim}a  \leq  b \vee {\sim}b$,
\item $a\wedge c \leq {\sim} a\vee b$ \ if and only if \ $c\leq a\rightarrow b$,
\item $(a\wedge b)\rightarrow c = a \rightarrow (b\rightarrow c)$.
\end{enumerate}
Nelson algebras provide models for constructive 
logic with strong negation, as shown by H.~Rasiowa \cite{Rasiowa74}.

In each Nelson algebra, an operation $\neg$ can be defined as
$\neg a = a \to 0$. The operation $\neg$ is called \emph{weak negation}.
A Nelson algebra $\mathbb{A}$ is \emph{semi-simple} if $a \vee \neg a = 1$ for all $a \in A$. It is well
known that semi-simple Nelson algebras coincide with three-valued {\L}ukasiewicz
algebras and regular double Stone algebras. Representations of semi-simple Nelson algebras,
three-valued {\L}ukasiewicz algebras, and regular double Stone algebras in terms of
rough sets determined by equivalences are found in 
\cite{Comer93, GeWa92, Iturrioz99, pagliani2008geometry, PomPom88}.

We proved in \cite{JarRad11} that in the case rough approximations are determined by a
quasiorder, the bounded distributive lattice $\mathit{RS}$ forms a Nelson algebra. 
We denote this Nelson algebra by $\mathbb{RS}$, and the operations are defined by: 
\begin{align*}
 (X^\DOWN,X^\UP) \vee  (Y^\DOWN,Y^\UP)   & = (X^\DOWN \cup Y^\DOWN, X^\UP \cup Y^\UP), \\
 (X^\DOWN,X^\UP) \wedge (Y^\DOWN,Y^\UP) & = (X^\DOWN \cap Y^\DOWN, X^\UP \cap Y^\UP), \\
 {\sim}(X^\DOWN,X^\UP) &= (-X^\UP, -X^\DOWN), \\
  (X^\DOWN,X^\UP)  \to  (Y^\DOWN,Y^\UP)  &= ( (-X^\DOWN \cup Y^\DOWN)^\DOWN, -X^\DOWN \cup Y^\UP), 
\end{align*}
where $-X$ denotes the set-theoretical complement $U \setminus X$ of the subset $X \subseteq U$.
The $0$-element is $(\emptyset,\emptyset)$ and $(U,U)$ is the 1-element (see also \cite{JPR12}).
We showed in \cite{JarRad11} that if $\mathbb{A}$ is a 
Nelson algebra defined on an algebraic lattice, then there exists a set $U$ and a quasiorder 
$R$ on $U$ such that $\mathbb{A}$ and the Nelson algebra $\mathbb{RS}$ are isomorphic.
Note that an \emph{algebraic lattice} $L$ is a complete lattice with the property that 
any element of it is equal to the join of some compact elements of $L$.
In \cite{JPR12}, we proved an algebraic completeness theorem for Nelson logic in terms 
of finite rough set-based Nelson algebras determined by quasiorders.

\section{Monteiro spaces}

An \emph{Alexandrov topology} $\mathcal{T}$ on $X$ is a topology in which an arbitrary intersection of 
open sets is open, or equivalently, every point $x \in X$ has the least neighbourhood $N(x) \in \mathcal{T}$.
For an Alexandrov topology $\mathcal{T}$, the least neighbourhood of a point $x$ is
$N(x) = \bigcap \{ B \in \mathcal{T} \mid x \in B \}$. We denote by $\mathcal{C}$ and $\mathcal{I}$ the
closure and the interior operators of $\mathcal{T}$, respectively. Then, $\mathcal{T} = \{ \mathcal{I}(B) \mid B \subseteq X\}$.
Additionally, $\mathcal{B}_\mathcal{T} = \{ N(x) \mid x \in X\}$ forms a \emph{smallest base} of $\mathcal{T}$,
implying that for all $B \in \mathcal{T}$, $B = \bigcup_{x \in B} N(x)$.
Note that a \emph{complete ring of sets} means exactly the same thing as Alexandrov topology
\cite{Alex37,Birk37}. For an Alexandrov topology $\mathcal{T}$ on $X$,  we may define a quasiorder 
$R_\mathcal{T}$ on $X$ by $x \, R_\mathcal{T} \, y$ if and only if $y \in N(x)$.

On the other hand, let $R$ be a quosiorder on $X$. The set of all $R$-closed subsets of $X$ forms an Alexandrov topology
$\mathcal{T}_R$, meaning that $B \in \mathcal{T}_R$ if and only if $x \in B$ and $x \, R \, y$ imply $y \in B$. 
Since the set $R(x)$ of $R$-successors is $R$-closed, we have $N(x) = R(x)$ in $\mathcal{T}_R$. 
In addition, $\mathcal{I}(B) = \{ x \in X \mid R(x) \subseteq B\} = B^\DOWN$ and
$\mathcal{C}(B) = \{ x \in X \mid R(x) \cap B \neq \emptyset \} = B^\UP$ for any $B \subseteq X$.

The correspondences $\mathcal{T} \mapsto R_\mathcal{T}$ and $R \mapsto \mathcal{T}_R$ are
mutually inverse bijections  between the class of all Alexandrov topologies and the class of the quasiorders 
on the set $X$.  In addition, it is known that a quasiorder $R$ is a partial order if and only if 
$\mathcal{T}_R$ satisfies the $\mathrm{T_0}$-separation axiom, that is, for any two different points 
$x$ and $y$, there is an open set which contains one of these points and not the other. 
Topologies satisfying the $\mathrm{T_0}$-separation axiom are called the $\mathrm{T_0}$-\emph{spaces}.
Therefore, there is a one-to-one correspondence between partial orders and Alexandrov topologies that are
$\mathrm{T_0}$-spaces.

For each topology $\mathcal{T}$ on $X$, the lattice ($\mathcal{T},\subseteq)$ forms a Heyting algebra 
such that the relative pseudo\-complement of $B,C \in \mathcal{T}$ is $B \Rightarrow C = \mathcal{I}(-B \cup C)$. 
In particular, for a quasiorder $R$, the relative pseudo\-complement in $\mathcal{T}_R$ can be expressed as
\[ 
  B \Rightarrow C = \{ x \in X \mid x \, R \, y \text{ and } y \in B \text{ imply } y \in C \}. 
\]

Let $(X,\leq,g)$ be a structure such that $(X,\leq)$ is a partially ordered set and $g$ is 
a map on $X$ satisfying the following conditions for all $x,y$:
\begin{enumerate}[({J}1)] \label{Page:Js}
 \item if $x \leq y$, then $g(y) \leq g(x)$,
 \item $g(g(x)) = g(x)$,
 \item $x \leq g(x)$ or $g(x) \leq x$,
 \item if $x,y \leq g(x),g(y)$, then there is $z \in X$ such that $x,y \leq z \leq g(x),g(y)$.
\end{enumerate}
According to D.~Vakarelov \cite{Vaka77}, these systems are called \emph{Monteiro spaces}, because
A.~Monteiro was the first who introduced them in \cite{Mont63a}.

Let $\leq$ be a partial order on $X$. It is typical that $\leq$-closed sets are called \emph{upsets}. We denote
by $\mathcal{U}(X)$ the set of all upsets of $X$. By the above, $\mathcal{U}(X)$ forms a $\mathrm{T_0}$-space. 

As proved by D.~Vakarelov \cite{Vaka77}, each Monteiro space $\mathcal{M} = (X,\leq,g)$ defines a Nelson algebra
\[ \mathbb{N}_\mathcal{M} = (\mathcal{U}(X),\vee,\wedge,\to,\sim,0,1) ,\]
where the operations are defined by:
\begin{align*}
 & 0 = \emptyset,&           & 1 = X, \\
 & A \vee B = A \cup B,&     & A \wedge B = A \cap B, \\
 & {\sim} A = \{ x \in X \mid g(x) \notin A \},&  &A \to B = A \Rightarrow ({\sim} A \cup B);
\end{align*}
note that the operation $\Rightarrow$ is defined in $\mathcal{U}(X)$ by:
\[ B \Rightarrow C = \{ x \in X \mid x \leq y \text{ and } y \in B \text{ imply } y \in C \}. \]
 
Let $\mathbb{A} = (A,\vee,\wedge,\to,{\sim},0,1)$ be a Nelson algebra. We denote by
$\mathcal{F}_p$ the set of prime filters of $\mathbb{A}$. We define for any
$P \in \mathcal{F}_P$ the set of elements
\[ g(P) = \{ x \in A \mid {\sim}x \notin P\}.
\]
The set $g(P)$ is known to be a prime filter of $A$, and the mapping $g$ on $\mathcal{F}_P$ satisfies
the conditions (J1)--(J4) with respect to the set-inclusion order (see also \cite{Cign86}). 
Thus, the structure $\mathcal{M} = (\mathcal{F}_p,\subseteq,g)$ is a Monteiro space and it determines
a Nelson algebra 
\[
\mathbb{N}_\mathcal{M} = (\mathcal{U}(\mathcal{F}_p), \cup, \cap, \to, {\sim}, \emptyset, \mathcal{F}_p).
\]

For any $x \in A$, we define a set of prime filters as
\[ h(x) = \{ P \in \mathcal{F}_p \mid x \in P \}. \]
If $P \in h(x)$ and $P \subseteq Q$, then $x \in P \subseteq Q$, that is, $Q \in h(x)$. 
Therefore, $h(x) \in \mathcal{U}(\mathcal{F}_p)$. Because $(A,\leq)$ is a distributive
a lattice, for all $x \neq y$, there exists a prime filter $P$ such that $x \in P$ and
$y \notin P$, or $x \notin P$ and $y \in P$ by the well-known ``prime filter theorem'' 
of distributive lattices´´.  This means that $h(x) \neq h(y)$, and
hence $h$ is an injection from $A$ to $\mathcal{U}(\mathcal{F}_p)$.

Next we will show that $h$ is a Nelson-algebra homomorphism:

\begin{itemize}
 \item $h(0) = \emptyset$, because prime filters must be proper filters. Therefore, $0$
does not belong to any prime filter.

\item $h(1) = \mathcal{F}_p$, because $1$ must belong to all prime filters.

\item $P \in h(x \vee y) \iff x \vee y \in P \iff x \in P \text{ or } y \in P \iff
P \in h(x) \text{ or } P \in h(y) \iff P \in h(x) \cup h(y)$.

\item $P \in h(x \wedge y) \iff x \wedge y \in P \iff x \in P \text{ and } y \in P \iff
P \in h(x) \text{ and }  P \in h(y) \iff P \in h(x) \cap h(y)$.

\item $P \in h({\sim}x) \iff {\sim} x \in P \iff x \notin g(P) \iff g(P) \notin h(x) \iff
P \in {\sim}h(x)$.
\end{itemize}

\noindent%
D.~Vakarelov \cite{Vaka77} has proved that for any $P \in \mathcal{F}_p$, $a \to b \in P$
if and only if for all $Q \in \mathcal{F}_p$, $P \subseteq Q$, $a \in Q$, and $a \in g(Q)$
imply $b \in Q$. Therefore, 
\begin{itemize}
\item $P \in h(x \to y) \iff x \to y \in P \iff$ for all $Q \in \mathcal{F}_p$, 
$P \subseteq Q$, $x \in Q$, and $x \in g(Q)$ imply $y \in Q \iff$ 
for all $Q \in \mathcal{F}_p$,  $P \subseteq Q$, $Q\in h(x)$, and $Q \notin {\sim}h(x)$ imply $Q \in h(y) \iff$ 
for all $Q \in \mathcal{F}_p$,  $P \subseteq Q$ and $Q\in h(x)$ imply $Q \in {\sim}h(x) \cup h(y)  \iff
P \in h(x) \Rightarrow ({\sim}h(x) \cup h(y)) = h(x) \to h(y)$.
\end{itemize}

We have now proved that $h$ is an injective homomorphism $A \to \mathcal{U}(\mathcal{F}_p)$.
Thus, $h$ is an Nelson-algebra embedding and we can write the following proposition that
appears already in \cite{Vaka77}.
\begin{proposition} Let $\mathbb{A}$ be a Nelson algebra. Then, $\mathbb{A}$ is
 isomorphic to a subalgebra of $\mathbb{N}_\mathcal{M}$, where $\mathcal{M}$ is
the Monteiro space $(\mathcal{F}_p,\subseteq,g)$.
\end{proposition}

For a Nelson algebra $\mathbb{A}$, the family of sets  $\mathcal{U}(\mathcal{F}_p)$
is an Alexandrov topology. Therefore, $(\mathcal{U}(\mathcal{F}_p),\subseteq)$ forms an algebraic lattice. 
As we already noted, we showed in \cite{JarRad11} that if $\mathbb{A}$ is a Nelson algebra such that 
its underlying lattice is algebraic, then there exists a universe $U$ and a quasiorder $R$ on $U$ such that 
$\mathbb{A} \cong \mathbb{RS}$. This means that $\mathbb{N}_\mathcal{M}$, where $\mathcal{M}$ is
the Monteiro space $(\mathcal{F}_p,\subseteq,g)$, is isomorphic to
some rough set Nelson algebra $\mathbb{RS}$, and let us denote by
$\varphi$ this isomorphism in question.

It is now obvious that the mapping $\varphi \circ h$ is an embedding
from $\mathbb{A}$ to $\mathbb{RS}$, and we can write the following
theorem.

\begin{theorem} \label{Thm:NelsonRepresentation}
Let $\mathbb{A}$ be a Nelson algebra. Then, there
exists a set $U$ and a quasiorder $R$ on $U$ such that $\mathbb{A}$
is isomorphic to a subalgebra of $\mathbb{RS}$.
\end{theorem}

\section{Alexandrov spaces and rough sets}

Let $\mathbb{A} = (A,\vee,\wedge,{\sim},\to,0, 1)$ be a Nelson algebra
such that the lattice $(A,\leq)$ is algebraic.
In the lattice $(A,\leq)$, each element of $A$ can be represented
as the join of completely join-irreducible elements $\mathcal{J}$ below it
(see \cite{JarRad11}). 
Let us define an order $\triangleleft$ on $\mathcal{J}$ by setting
\[ x \triangleleft y \iff y \leq x \text{ in } A.\]
Let $\mathcal{U(J)}$ be the set of upsets with respect to $\triangleleft$.
Then, $\mathcal{U(J)}$ is an Alexandrov topology and a $\mathrm{T_0}$-space
(because $\triangleleft$ is a partial order on $\mathcal{J}$).
It is now clear that for all $x,y \in \mathcal{J}$,
\[ x \leq y \iff N(x) \subseteq N(y);\]
note that $N(x) = \{ y \in \mathcal{J} \mid x \triangleleft y \}$.

\medskip

It is known that the set of completely join-irreducible elements of
$\mathcal{U(J)}$ is $\mathcal{B} = \{ N(x) \mid x \in \mathcal{J} \}$.
We define a map 
\[ \varphi \colon \mathcal{J} \to \mathcal{B}, x \mapsto N(x).\]
Clearly, this map is an order-isomorphism between $(\mathcal{J},\leq)$
and $(\mathcal{B},\subseteq)$. This means that $\varphi$ can be
canonically extended to a lattice-isomorphism $\Phi \colon A \to \mathcal{U(J)}$
by
\begin{align*}
\Phi(x) & = \bigcup \{ \varphi(j) \mid j \in \mathcal{J} \text{ and } j \leq x \} \\
& = \bigcup \{ N(j) \mid j \in \mathcal{J} \text{ and } j \leq x \}.
\end{align*}
Obviously, $\Phi(0) = \emptyset$, $\Phi(1) = \mathcal{J}$, and since $A$
and $\mathcal{U(J)}$ are Heyting algebras, the relative pseudo\-complement
satisfies $\Phi(x \Rightarrow y) =  \Phi(x) \Rightarrow \Phi(y)$. This
is because the relative pseudo\-complement is unique in the sense that it
depends only on the order of the Heyting algebra in question, and now
the ordered sets $(A,\leq)$ and $(\mathcal{U(J)},\subseteq)$ are isomorphic.

Note that for all $x \in A$ and $j \in \mathcal{J}$,
\[ j \in \Phi(x) \iff j \leq x.\]
Namely, if $j \in \Phi(x)$, then $j \in N(k)$ for some $k \in \mathcal{J}$
such that $k \leq x$. Thus, $k \triangleleft j$ and $j \leq k$, which
give $j \leq x$. On the other hand, if $j \leq x$, then $j \in N(j)$
gives $j \in \Phi(x)$.

We may now define a map $g \colon \mathcal{J} \to \mathcal{J}$ by
setting 
\[
g(j) = \bigwedge \{ x \in A \mid x \nleq {\sim} j \}.   
\]
By our work \cite{JarRad11}, $(\mathcal{J},{\triangleleft},g)$
forms a Monteiro space. Thus, the structure
\[
  (\mathcal{U(J)}, \cup, \cap, \to, {\sim}, \emptyset, \mathcal{J})
\]
is a Nelson algebra. Because the operation $\to$ is defined in terms of 
$\Rightarrow$ and $\sim$, to show that this is isomorphic to $\mathbb{A}$, it
suffices to show that
\[ \Phi({\sim}x) = {\sim}\Phi(x) \]
for all $x \in A$. 
Now,
\[
  \Phi({\sim}x) = \{ j \in \mathcal{J} \mid j \leq {\sim} x\}.
\]
On the other hand, by the definition of $\sim$ in $\mathcal{U(J)}$, we have:
\begin{align*} 
\sim \Phi(x) & = \{ j \in \mathcal{J} \mid g(j) \notin \Phi(x) \} \\
 & = \{ j \in \mathcal{J} \mid g(j) \nleq x \}.
\end{align*}
We have noted in \cite{JarRad11} that for all $x \in A$ and
$j \in \mathcal{J}$, $g(j) \nleq x$ iff $j \leq {\sim} x$.
Therefore, we have proved the following theorem.

\begin{theorem}
Let $\mathbb{A}$ be a Nelson algebra such that the lattice $(A,\leq)$ is algebraic.
Then,  $\mathbb{A}$ and $(\mathcal{U(J)}, \cup, \cap, \to, {\sim}, \emptyset, \mathcal{J})$ 
are isomorphic. 
\end{theorem}

\medskip

We end this work by presenting the following theorem that
shows how certain structures studied in this work can be
considered as equivalent structures.

\begin{theorem} \label{Thm:Main}
The following structures determine each other  ``up-to-isomorphism'' 
(and hence they can be considered equivalent):
 \begin{enumerate}[\rm (i)]
  \item Rough sets induced by quasiorders;
  \item Nelson algebras defined on algebraic lattices;
  \item Nelson algebras defined on $\mathrm{T_0}$-spaces that are Alexandrov topologies;
  \item Monteiro spaces.
 \end{enumerate}
\end{theorem}

\begin{proof}
Cases (i), (ii) and (iii) are all ``equivalent'', as we have seen: each
Nelson algebra defined on an algebraic lattice can be represented up to
isomorphism as (i) and (iii). Each Monteiro space induces a Nelson algebra
defined on an algebraic lattice, and each Nelson algebra defined on an
algebraic lattice induces a Monteiro space that determines an Alexandrov 
topology Nelson algebra isomorphic to the original algebra.
Thus, (ii) and (iv) are equivalent.
\end{proof}

In our next example, we illustrate the isomorphisms between different structures.

\begin{example}
As we already noted, for a Nelson algebra $\mathbb{A}$ defined on an algebraic lattice, its each element can be represented
as the join of completely join-irreducible elements $\mathcal{J}$ below it. Therefore, concerning the structure of $\mathbb{A}$,
the essential thing is how its completely join-irreducible elements are related to each other. In addition, isomorphisms between
Nelson algebras defined on algebraic lattices are completely defined by maps on completely join-irreducible elements.

The map $g \colon \mathcal{J} \to \mathcal{J}$, defined by 
\begin{equation} \label{Eq:DefOfMap} \tag{$\star$}
 g(j) = \bigwedge \{ x \in A \mid x \nleq {\sim} j \},
\end{equation}
satisfies conditions (J1)--(J4) as noted in page~\pageref{Page:Js}.
Particularly, we have by (J3) that $j \leq g(j)$ or $g(j) \leq j$ for any $j \in \mathcal{J}$. 
We define for every $j \in \mathcal{J}$ a ``representative'' $\rho(j)$ by
\[
\rho(j) = \left \{
\begin{array}{ll}
j    & \mbox{ if $j \leq g(j)$}\\
g(j)  & \mbox{ otherwise.}
\end{array}
\right .
\]
In terms of $\rho$, we define a quasiorder $R$ on $U = \mathcal{J}$ by setting $x \, R \, y \iff \rho(x) \leq \rho(y)$.

In \cite{JarRad11}, we showed that for this quasiorder $R$ on $\mathcal{J}$, 
$\mathbb{RS}$ and $\mathbb{A}$ are isomorphic Nelson algebras. If  $\mathcal{J}(RS)$
denotes the set of completely join-irreducible elements of $\mathit{RS}$, then the isomorphism
$\varphi \colon \mathcal{J} \to \mathcal{J}(RS)$ is defined by 
\[
\varphi(j) = \left \{
\begin{array}{ll}
(\emptyset,\{j\}^\UP) & \mbox{ if $j \leq g(j)$}\\
(R(j),R(j)^\UP)       & \mbox{ otherwise.}
\end{array}
\right .
\]

\begin{figure}[h]
\centering
\includegraphics[width=\textwidth]{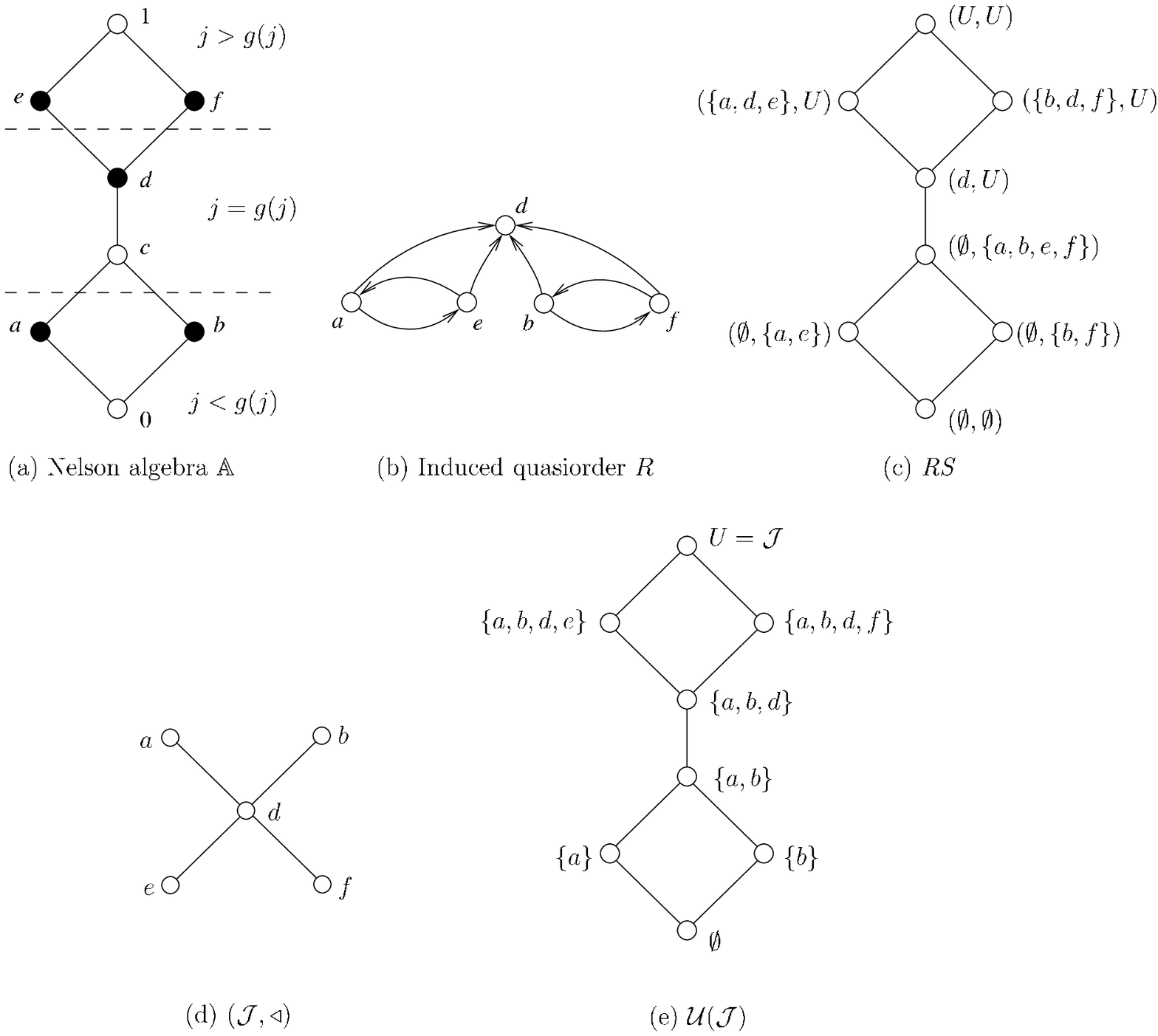}
\caption{\label{Fig:figure1}}
\end{figure}

Consider the Nelson algebra $\mathbb{A}$ of Figure~\ref{Fig:figure1}(a). Because it is finite, it
is trivially defined on an algebraic lattice. Suppose that the operation $\sim$ is defined by
${\sim}0 = 1$, ${\sim} a = f$, ${\sim} b = e$, and ${\sim} c = d$. The completely join-irreducible
elements $\mathcal{J}$ are marked by filled circles, and we have $g(a) = e$, $g(b) = f$, and $g(d) = d$.
The induced quasiorder on $U = \mathcal{J} = \{a,b,d,e,f\}$ is given in Figure~\ref{Fig:figure1}(b) and
the corresponding rough set structure $\mathit{RS}$ is depicted in Figure~\ref{Fig:figure1}(c).
Recall that the operation $\sim$ is defined in $\mathbb{RS}$ by ${\sim}(X^\DOWN,X^\UP) = (-X^\UP, -X^\DOWN)$.

On the other hand, the partially ordered set $(\mathcal{J},{\triangleleft})$ induced by $\mathbb{A}$ is given in Figure~\ref{Fig:figure1}(d).
The corresponding structure of its upsets $\mathcal{U(J)}$ can be seen in Figure~\ref{Fig:figure1}(e). 
If $\mathcal{B}$ denotes the set of completely join-irreducible elements of $\mathcal{U(J)}$, 
then the mapping $\psi \colon \mathcal{J} \to \mathcal{B}$ defined by $j \mapsto N(j)$ is
an order-isomorphisms, where $N(j)$ is the principal filter of $j$ with respect to $\triangleleft$. 
The map $g$ in the Monteiro space $(\mathcal{J}, {\triangleleft}, g)$
is defined by \eqref{Eq:DefOfMap}, and the operation $\sim$ in the 
Nelson algebra  $(\mathcal{U(J)}, \cup, \cap, \to, {\sim}, \emptyset, \mathcal{J})$ is given by
${\sim} X = \{ j \in \mathcal{J} \mid g(j) \notin X \}$.
\end{example}

\section{Some concluding remarks}
We end this paper by noting that there is a perfect analogy between the algebraic counterparts of 
classical logic and constructive logic with strong negation:
The basic algebraic structures of classical logic are Boolean algebras, and by the well-known
representation theorem of M.~H.~Stone, any Boolean algebra is isomorphic with a field of sets, that is, with a 
subalgebra of the Boolean algebra defined on a power set. Analogously, the algebraic 
counterparts of constructive logic with strong negation are Nelson algebras, and any Nelson algebra is 
isomorphic to a subalgebra of the Nelson algebra defined on a rough set lattice, according to 
Theorem~\ref{Thm:NelsonRepresentation}.

Theorem~\ref{Thm:Main} means that a rough set system determined by a quasiorder can be treated as a family of upsets of a partially ordered set, 
which is a well-studied structure in the literature. For any partially ordered set, its upsets form a complete lattice, but in our case, upsets 
form an Alexandrov $\mathrm{T_0}$-space provided with operations of a Nelson algebra. 
This also opens new approaches for further research concerning the representation of particular objects in the category of Nelson algebras.


\end{document}